\documentclass[11pt,twoside]{article}
\usepackage{amsfonts}
\usepackage{amsmath,amssymb}
\usepackage{multicol}
\usepackage{graphics}
\usepackage{cite}
\usepackage{stmaryrd}
\usepackage{mathrsfs} 

\usepackage{color}
\usepackage{graphicx}
\usepackage{float}  
\usepackage{subfig}
\usepackage{cleveref}
\crefname{equation}{}{}{}
\usepackage{comment}
\usepackage{cases}
\usepackage{indentfirst}

\newtheorem{theorem}{Theorem}[section]

\newtheorem{corollary}[theorem]{Corollary}

\newtheorem{remark}[theorem]{Remark}
\newtheorem{proposition}[theorem]{Proposition}

\newtheorem{assumption}[theorem]{Assumption}
\numberwithin{equation}{section}
\newenvironment{proof}[1][Proof]{\noindent\textbf{#1.} }{\hfill $\Box$}

\allowdisplaybreaks \numberwithin{equation}{section}
\makeatletter
\begin{document}
\author{Dongsheng Li, Yasheng Lyu \\
School of Mathematics and Statistics, Xi'an Jiaotong University,\\
 Xi'an, Shaanxi 710049, People's Republic
of China }
\title{\textbf{A new proof of H\"older estimates for the gradient of quasilinear elliptic equations} }
 \date{}
\maketitle

\begin{abstract}
In this paper, we give a new proof of H\"older estimates for the gradient of quasilinear elliptic equations, using a covering method inspired by the proof of Evans-Krylov theorem for fully nonlinear elliptic equations.
Moreover, H\"older estimates for the gradient of fully nonlinear elliptic equations are also obtained by the same method. 

\\
\textbf{AMS Subject Classification (2020)}: 
\end{abstract}

\section{Introduction}

For quasilinear elliptic equations in nondivergence form 
\begin{equation}\label{eqn1}
A^{ij}(x,u,Du)D_{ij}u+B(x,u,Du)=0\ \ \text{in}\ \Omega,
\end{equation}
H\"older estimates for the gradient $Du$ were first established by Ladyzhenskaya and Ural'tseva \cite{LU1,LU2,LU3}. 
These foundational estimates enable the application of Schauder theory to obtain higher-order estimates provided that the equation is sufficiently smooth.
Later, Trudinger \cite{Trudinger} established certain Harnack-type inequalities and gave an alternative proof of the gradient H\"older estimates. 

As a counterpart, for fully nonlinear elliptic equations, Evans \cite{Evans} and Krylov \cite{Krylov1,Krylov2} independently proved ‌the‌ Evans-Krylov theorem, establishing H\"older estimates for the Hessian $D^{2}u$.
This foundational result allows the application of Schauder theory to achieve  higher-order estimates provided that the equation is sufficiently smooth. 
Subsequently, Caffarelli and Cabr\'e \cite{Caffarelli} presented‌ a proof of ‌the‌ Evans-Krylov theorem using a covering method‌. 
Caffarelli and Silvestre \cite{Caffarelli2010,Caffarelli2011} established the Evans-Krylov theorem for nonlocal fully nonlinear equations and also derived  another proof of the‌ Evans-Krylov theorem. 

In this paper, we adapt the covering method \cite{Caffarelli} to provide a new proof of H\"older estimates for the gradient $Du$ of nonlinear elliptic equations.
Specifically, Caffarelli and Cabr\'e \cite{Caffarelli} covered $D^{2}u(B_{1})$ in the space of real symmetric $n\times n$ matrices with ‌a finite number of‌ balls of universal radius.
The number of required covering balls can be further reduced to cover $D^{2}u(B_{\delta})$.
By analyzing the oscillation decay of $D^{2}u$ as the radius of the domain ball diminishes‌, the H\"older continuity of $D^{2}u$ is established‌.
Throughout this paper, $\Omega$ is assumed to be a smooth bounded domain. 

Assume that the coefficients of equation \eqref{eqn1} satisfy 
\begin{equation}\label{eqn6}
\text{matrix}\ [A^{ij}]\ \text{is positive definite},\ 
A^{ij}\in C^{1},\ B\in C^{0}\ \ \text{in}\ \Omega\times\mathbb{R}\times\mathbb{R}^{n},
\end{equation}
which implies that there is a continuous function $\theta_{1}$ such that for any constant $\rho>0$,
\begin{equation}\label{eqn9}
\frac{\left|A^{ij}\right|+\left|D_{p_{k}}A^{ij}\right|+\left|D_{z}A^{ij}\right|+\left|D_{x_{k}}A^{ij}\right|+|B|}{\lambda}\leq\theta_{1}(\rho)
\end{equation}
in $\left\{(x,z,p)\in\Omega\times\mathbb{R}\times\mathbb{R}^{n}:\ |z|+|p|\leq\rho\right\}$, where $\lambda$ is the minimal eigenvalue of $[A^{ij}]$.

\begin{theorem}\label{thm1.11}
Let $u\in C^{1}(\Omega)\cap W^{2,2}(\Omega)$ satisfy the equation \eqref{eqn1}, and let the condition \eqref{eqn6} hold. 
Then we have H\"older seminorm estimates
\[
[Du]_{\alpha;\Omega_{d}}\leq Cd^{-\alpha},
\]
where $\Omega_{d}:=\{x\in\Omega:\ \operatorname{dist}(x,\partial\Omega)>d\}$,  $C=C\big(n,\|Du\|_{C^{0}(\Omega_{d\hspace{-0.15mm}/\hspace{-0.2mm}2})},\operatorname{diam}\Omega\big)$, and $\alpha\in(0,1)$ depending on $n$, $\operatorname{diam}\Omega$, $\theta_{1}$ and  $\|u\|_{C^{1}(\Omega_{d\hspace{-0.15mm}/\hspace{-0.2mm}2})}$.  
\end{theorem}

For fully nonlinear elliptic equations, H\"older estimates for the gradient $Du$ can be obtained using arguments similar to those in the case of quasilinear elliptic equations.
We will discuss this in Section 5.

\section{A distance on $\mathbb{R}^{n}$}

For any constant $\gamma\neq0$, define
\begin{equation}\label{eqn2.111}
\operatorname{dist}_{\gamma}(x_{1},x_{2}):=\sup_{e\in\mathbb{S}^{n-1}}\left|\gamma(x_{1}-x_{2})\cdot e+|x_{1}|^{2}-|x_{2}|^{2}\right|,\ \ \forall x_{1},\ x_{2}\in\mathbb{R}^{n},
\end{equation}
where $\mathbb{S}^{n-1}$ denotes the unit sphere in $\mathbb{R}^{n}$. 

\begin{proposition}\label{lem11}
$\operatorname{dist}_{\gamma}(\cdot,\cdot)$ is a distance on $\mathbb{R}^{n}$.
\end{proposition} 
\begin{proof} 
First, the symmetry $\operatorname{dist}_{\gamma}(x_{1},x_{2})=\operatorname{dist}_{\gamma}(x_{2},x_{1})\in[0,\infty)$ holds for all $x_{1}$, $x_{2}\in\mathbb{R}^{n}$.
Second, if $\operatorname{dist}_{\gamma}(x_{1},x_{2})=0$, then taking $e\perp(x_{1}-x_{2})$ implies $|x_{1}|=|x_{2}|$, 
and taking $e\parallel(x_{1}-x_{2})$ again yields $x_{1}=x_{2}$.
Third, for any $x_{1}$, $x_{2}$, $x_{3}\in\mathbb{R}^{n}$,
\begin{align*}
\operatorname{dist}_{\gamma}(x_{1},x_{2})&=\sup_{e\in\mathbb{S}^{n-1}}\left|\gamma(x_{1}-x_{3})\cdot e+|x_{1}|^{2}-|x_{3}|^{2}+\gamma(x_{3}-x_{2})\cdot e+|x_{3}|^{2}-|x_{2}|^{2}\right|\\
&\leq\operatorname{dist}_{\gamma}(x_{1},x_{3})+\operatorname{dist}_{\gamma}(x_{3},x_{2}). 
\end{align*} 
The proof of Proposition \ref{lem11} is complete.
\end{proof}
\par
\vspace{2mm}
In this paper, the ball of radius $R$ centered at $x_{0}$ with respect to the Euclidean distance is denoted by $B(x_{0},R)$.
The ball with respect to the distance $\operatorname{dist}_{\gamma}(\cdot,\cdot)$ is defined by 
\[
\mathcal{B}_{\gamma}(x_{0},R):=\left\{x\in\mathbb{R}^{n}:\ \operatorname{dist}_{\gamma}(x_{0},x)<R\right\},
\]
where $x_{0}$ is its center and $R$ is its radius.
For a set $\Sigma\subset\mathbb{R}^{n}$, its diameter with respect to the distance $\operatorname{dist}_{\gamma}(\cdot,\cdot)$ is defined by 
\[
\operatorname{diam}_{\gamma}(\Sigma):=\sup_{x_{1},x_{2}\in\Sigma}\operatorname{dist}_{\gamma}(x_{1},x_{2}).
\]

The distance $\operatorname{dist}_{\gamma}(\cdot,\cdot)$ lacks positive homogeneity, translation invariance and comparability to the Euclidean distance, while it has the following property. 
\begin{proposition}\label{pro2.2}
It holds that
\[
|\gamma||x_{1}-x_{2}|\leq\operatorname{dist}_{\gamma}(x_{1},x_{2})\leq(|\gamma|+|x_{1}|+|x_{2}|)|x_{1}-x_{2}|,\ \ \forall x_{1},\ x_{2}\in\mathbb{R}^{n}.
\]
\end{proposition}
\begin{proof}
Without loss of generality, assume $|x_{1}|\geq|x_{2}|$. 
It follows directly that
\[
\operatorname{dist}_{\gamma}(x_{1},x_{2})\leq|\gamma||x_{1}-x_{2}|+(x_{1}+x_{2})\cdot(x_{1}-x_{2})\leq(|\gamma|+|x_{1}|+|x_{2}|)|x_{1}-x_{2}|.
\]
Furthermore, we have  
\begin{equation*}
\operatorname{dist}_{\gamma}(x_{1},x_{2})\geq|\gamma|(x_{1}-x_{2})\cdot \tfrac{x_{1}-x_{2}}{|x_{1}-x_{2}|}+|x_{1}|^{2}-|x_{2}|^{2}\geq|\gamma||x_{1}-x_{2}|.
\end{equation*}
The proof of Proposition \ref{pro2.2} is complete.
\end{proof}
\par
\vspace{2mm}
In fact, the distance $\operatorname{dist}_{\gamma}(\cdot,\cdot)$ is comparable to the Euclidean distance in any compact set.
\begin{corollary}\label{cor2.3}
There exist two positive constants $\kappa_{1}$ and $\kappa_{2}$ depending on $R\in(0,\infty)$ such that for any $x$, $y\in B(0,R)$,
\[
\kappa_{1}\operatorname{dist}_{\gamma}(x,y)\leq|x-y|\leq\kappa_{2}\operatorname{dist}_{\gamma}(x,y).
\]
\end{corollary}
\begin{corollary}\label{cor2.4}
There exist two positive constants $\kappa_{1}$ and $\kappa_{2}$ depending on $R\in(0,\infty)$ such that for any $\mathcal{B}_{\gamma}(x_{0},\mu)\subset B(0,R)$,
\[
B(x_{0},\kappa_{1}\mu)\subset\mathcal{B}_{\gamma}(x_{0},\mu)\subset B(x_{0},\kappa_{2}\mu).
\]
\end{corollary}

\section{A particular type of vector-valued function}

In this section we will use the covering method to get H\"older estimates of a vector-valued function, provided that certain combinations of its components are generalized subsolutions of linear elliptic equations in divergence form.

\begin{assumption}\label{lem12}
Let $\psi(x)$ be an $n$-dimensional vector-valued function in $\Omega$, and $\psi^{k}(x)$ denote its $k$-th component for all $1\leq k\leq n$.
Let $\Omega_{*}$ be a fixed subset of $\Omega$,
\[
\psi\in C^{0}(\overline{\Omega_{*}})\cap W^{1,2}(\Omega_{*})\ \ \text{and}\ \ M:=\max\Big\{\sup_{\Omega_{*}}|\psi|,1\Big\}.
\]
Assume that there exists a constant $\gamma_{*}\geq4nM$
such that for any $k\in\{1,\dots,n\}$, the functions 
\[
v_{\pm}^{k}(x):=\pm\gamma_{*}\psi^{k}(x)+|\psi(x)|^{2}\ \ \text{in}\ \overline{\Omega_{*}}
\]
are subsolutions of linear elliptic equations in divergence form
\[
D_{i}\big(a^{ij}(x)D_{j}u+b^{i}(x)u\big)+c^{i}(x)D_{i}u+d(x)u=g+D_{i}f^{i}\ \ \text{in}\ \Omega_{*},
\]
in the generalized sense.
Here, the coefficients satisfy the following conditions in $\Omega_{*}$: 
\[
a^{ij},\ b^{i},\ c^{i},\ d,\ f^{i},\ g\ \text{are measurable};\ \ \frac{\Lambda}{\lambda}\leq\sigma_{*};\ \ \frac{|b^{i}|+|c^{i}|+|d|+|f^{i}|+|g|}{\lambda}\leq\nu,
\]
where $\sigma_{*}$ and $\nu$ are constants independent of $k$, and $\lambda$, $\Lambda$ denote the minimal and maximal eigenvalues of the positive definite matrix $[a^{ij}]$, respectively.
\end{assumption}

\begin{theorem}\label{lem2.1}
Let $B(y,R)$ be any ball contained in $\Omega_{*}$;   
let   
\[
\mu<\operatorname{diam}_{\gamma_{*}}(\psi(B(y,R)))\leq2\mu\ \ \text{and}\ \ 2\mu\geq R; 
\]
let $\psi(B(y,R))$ be covered by $N\geq2$ balls $\mathcal{B}_{\gamma_{*}}^{1},\dots,\mathcal{B}_{\gamma_{*}}^{N}$ of radius $\varepsilon\mu$ with respect to the distance $\operatorname{dist}_{\gamma_{*}}(\cdot,\cdot)$; 
and let Assumption \textsl{\ref{lem12}} hold.
Then there exist two small positive constants $\varepsilon_{0}$ and $\delta$, both independent of $y$, $R$ and $\mu$, such that, if in addition $\varepsilon\leq\varepsilon_{0}$ then $\psi(B(y,\delta R))$ can be covered by $N-1$ balls among $\mathcal{B}_{\gamma_{*}}^{1},\dots,\mathcal{B}_{\gamma_{*}}^{N}$.
Furthermore, both constants $\varepsilon_{0}$ and $\delta$ depend only on $n$, $\psi(\overline{\Omega_{*}})$, $\gamma_{*}$, $\sigma_{*}$, $\nu$ and $\operatorname{diam}\Omega_{*}$.
\end{theorem}
\begin{remark}
Since the distance $\operatorname{dist}_{\gamma_{*}}(\cdot,\cdot)$ does not have positive homogeneity, we cannot use normalization.
\end{remark}
\begin{proof}
Let $x_{s}\in B(y,R)$ be such that $\mathcal{B}^{s}_{\gamma_{*}}\subset \mathcal{B}_{\gamma_{*}}(M_{s},2\varepsilon\mu)$ for all $s\in\{1,\dots,N\}$, where
\[
M_{s}:=\psi(x_{s}).
\] 
For convenience, denote
\[
c_{0}:=\frac{\mu}{32\sqrt{n}}.
\]
If we take $\varepsilon_{0}\leq1/(128\sqrt{n})$, then the balls  $\{\mathcal{B}_{\gamma_{*}}(M_{s},c_{0}/2)\}_{s=1}^{N}$ cover $\psi(B(y,R))$.

Let $\mathcal{B}_{\gamma_{*}}$ be any ball of radius $2\mu$ with respect to the  distance $\operatorname{dist}_{\gamma_{*}}(\cdot,\cdot)$ that intersects  $\psi(\overline{\Omega_{*}})$.
Let $N'\in(0,+\infty]$ denote the supremum of the number of points in $\mathcal{B}_{\gamma_{*}}$ such that the distance $\operatorname{dist}_{\gamma_{*}}(\cdot,\cdot)$ between any two of them is at least $c_{0}/2$.
We claim that
\begin{equation}\label{claim3.1}
\text{$N'$ is finite and depends only on $n$, $\psi(\overline{\Omega_{*}})$ and $\gamma_{*}$.}
\end{equation}
Clearly, $\mathcal{B}_{\gamma_{*}}$ is contained in a compact set of $\mathbb{R}^{n}$ which depends only on $n$ and $\psi(\overline{\Omega_{*}})$.
Observe that the number of points in the unit ball $B_{1}\subset\mathbb{R}^{n}$, with the property that the Euclidean distance between any two of them is at least $h$, is finite and its maximum depends only on $n$ and $h$.
Then claim \eqref{claim3.1} follows from Corollaries \ref{cor2.3} and \ref{cor2.4}.
Consequently, there exists a constant $\widetilde{N}\leq \min\{N',N\}$ such that the balls $\{\mathcal{B}_{\gamma_{*}}(M_{s_{i}},c_{0})\}_{i=1}^{\widetilde{N}}$ cover $\psi(B(y,R))$.

For any $j\in\{1,\dots,n\}$, define
\[
V_{\pm}^{j}:=\sup_{B(y,R)}v_{\pm}^{j}\ \ \text{and}\ \ w_{\pm}^{j}:=V_{\pm}^{j}-v_{\pm}^{j}\ \ \text{in}\ B(y,R).
\]
By Assumption \ref{lem12} and the weak Harnack inequality (see Theorem 8.18 of \cite{GT}), we obtain that for any $\tau<1/2$ and $j\in\{1,\dots,n\}$,
\begin{equation}\label{eqn2.10}
(\tau R)^{-n}\int_{B(y,2\tau R)}w_{\pm}^{j}(x)\ dx\leq K\left(\inf_{B(y,\tau R)}w_{\pm}^{j}+\tau R\right),
\end{equation}
where $K$ depends on $n$, $\sigma_{*}$, $\nu$, $\operatorname{diam}\Omega_{*}$ and  $\|\psi\|_{C^{0}(\overline{\Omega_{*}})}$.
By the condition $2\mu\geq R$, there exists a constant $\delta\in(0,1/2)$ depending on 
$n$, $\psi(\overline{\Omega_{*}})$, $\gamma_{*}$, $\sigma_{*}$, $\nu$ and $\operatorname{diam}\Omega_{*}$, such that 
\begin{equation}\label{eqn2.11}
\frac{2^{n}|B_{1}|}{N'}c_{0}-K\delta R\geq\frac{2^{n-1}|B_{1}|}{N'}c_{0}. 
\end{equation}

There exists some $M_{s_{i}}$, say $M_{1}$, such that
\begin{equation}\label{eqn2.12}
\left|\psi^{-1}(\mathcal{B}_{\gamma_{*}}(M_{1},c_{0}))\cap B(y,2\delta R)\right|\geq\frac{|B(y,2\delta R)|}{\widetilde{N}}\geq\frac{|B(y,2\delta R)|}{N'}.
\end{equation}
Since the balls $\{\mathcal{B}_{\gamma_{*}}(M_{s},2\varepsilon\mu)\}_{s=1}^{N}$ cover $\psi(B(y,R))$ and  $\varepsilon_{0}\leq1/8$, there must exist another $M_{s}$, say $M_{2}$, such that  
\begin{equation}\label{eqn2.1}
\operatorname{dist}_{\gamma_{*}}(\psi(x_{1}),\psi(x_{2}))=\operatorname{dist}_{\gamma_{*}}(M_{1},M_{2})\geq\frac{\mu}{4}.
\end{equation}
Let $k\in\{1,\dots,n\}$ be the index satisfying 
\[
\left|\psi^{k}(x_{2})-\psi^{k}(x_{1})\right|=\max_{1\leq l\leq n}\left|\psi^{l}(x_{2})-\psi^{l}(x_{1})\right|.
\]
By \eqref{eqn2.1}, \eqref{eqn2.111} and the condition $\gamma_{*}\geq4nM$, we deduce 
\begin{equation}\label{eqn9.5}
\begin{split}
\frac{\mu}{4}&\leq\gamma_{*}|\psi(x_{2})-\psi(x_{1})|+\sum_{l=1}^{n}\left|\Big(\psi^{l}(x_{2})+\psi^{l}(x_{1})\Big)\Big(\psi^{l}(x_{2})-\psi^{l}(x_{1})\Big)\right|\\
&\leq\gamma_{*}\sqrt{n}\left|\psi^{k}(x_{2})-\psi^{k}(x_{1})\right|+ 2nM\Big|\psi^{k}(x_{2})-\psi^{k}(x_{1})\Big|\\
&\leq(2\sqrt{n}+1)(\gamma_{*}-2nM)\left|\psi^{k}(x_{2})-\psi^{k}(x_{1})\right|\\
&\leq4\sqrt{n}\left(\tilde{v}^{k}(x_{2})-\tilde{v}^{k}(x_{1})\right),
\end{split}
\end{equation}
where 
\[
\tilde{v}^{k}:=\gamma_{*} \psi^{k}\times\operatorname{sign}\Big(\psi^{k}(x_{2})-\psi^{k}(x_{1})\Big)+|\psi|^{2}\ \ \text{in}\ B(y,R). 
\]
It follows from \eqref{eqn2.111} that 
\begin{equation}\label{eqn9.1}
\left|\tilde{v}^{k}(x)-\tilde{v}^{k}(y)\right|\leq\operatorname{dist}_{\gamma_{*}}(\psi(x),\psi(y)),\ \ \forall x,\ y.
\end{equation}
Let $\widetilde{V}^{k}$ denote the supremum of $\tilde{v}^{k}$ over $B(y,R)$. 
From \eqref{eqn9.5}, we obtain  
\[
\widetilde{V}^{k}\geq\tilde{v}^{k}(x_{2})\geq\tilde{v}^{k}(x_{1})+\frac{\mu}{16\sqrt{n}}= \tilde{v}^{k}(x_{1})+2c_{0},
\]
which, combined with \eqref{eqn9.1}, implies  
\begin{equation}\label{eqn2.4}
\widetilde{V}^{k}-\tilde{v}^{k}(x)\geq c_{0},\ \ \forall x\in\psi^{-1}(\mathcal{B}_{\gamma_{*}}(M_{1},c_{0}))\cap B(y,2\delta R). 
\end{equation}
It is clear that either $\tilde{v}^{k}=v^{k}_{+}$ or $\tilde{v}^{k}=v^{k}_{-}$. 
By virtue of \eqref{eqn2.10}, \eqref{eqn2.12} and \eqref{eqn2.4}, one gets 
\[
\inf_{B(y,\delta R)}\left\{\widetilde{V}^{k}-\tilde{v}^{k}\right\}\geq \frac{1}{K}\left[(\tau R)^{-n}\frac{|B(y,2\tau R)|}{N'}c_{0}-K\tau R\right].
\]
Combining this with \eqref{eqn2.11} yields 
\[
\inf_{B(y,\delta R)}\left\{\widetilde{V}^{k}-\tilde{v}^{k}\right\}\geq \frac{2^{n-6}|B_{1}|}{\sqrt{n}KN'}\mu.
\]
Let 
\[
\varepsilon_{0}:=\min\left\{\frac{1}{5}\frac{2^{n-6}|B_{1}|}{\sqrt{n}KN'},\frac{1}{128\sqrt{n}}\right\},
\]
which depends only on $n$, $\psi(\overline{\Omega_{*}})$, $\gamma_{*}$, $\sigma_{*}$, $\nu$ and $\operatorname{diam}\Omega_{*}$. 
Thus   
\begin{equation}\label{eqn2.6}
\inf_{B(y,\delta R)}\left\{\widetilde{V}^{k}-\tilde{v}^{k}\right\}\geq 5\varepsilon\mu.
\end{equation}

As $\widetilde{V}^{k}$ is the supremum of $\tilde{v}^{k}$ over $B(y,R)$, there exists a point $\bar{x}\in B(y,R)$ such that 
\[
\big|\widetilde{V}^{k}-\tilde{v}^{k}(\bar{x})\big|<\varepsilon\mu. 
\]
Since the balls $\{\mathcal{B}_{\gamma_{*}}(M_{s},2\varepsilon\mu)\}_{s=1}^{N}$ cover $\psi(B(y,R))$, there must exist an index $s\in\{1,\dots,N\}$ with 
\[
\psi(\bar{x})\in \mathcal{B}_{\gamma_{*}}(M_{s},2\varepsilon\mu). 
\]
Thus together with \eqref{eqn9.1}, we get 
\begin{equation}\label{eqn2.7}
\widetilde{V}^{k}-\tilde{v}^{k}(x_{s})<\varepsilon\mu+\left|\tilde{v}^{k}(\bar{x})-\tilde{v}^{k}(x_{s})\right|\leq\varepsilon\mu+\operatorname{dist}_{\gamma_{*}}(\psi(\bar{x}),M_{s})<3\varepsilon\mu.
\end{equation}

Assume by contradiction that $\psi(B(y,\delta R))\cap \mathcal{B}_{\gamma_{*}}^{s}$ is not empty.
This implies there exists a point $\tilde{x}\in B(y,\delta R)$ satisfying 
\[
\psi(\tilde{x})\in\mathcal{B}_{\gamma_{*}}(M_{s},2\varepsilon\mu).
\]
Then one gets from \eqref{eqn9.1} that 
\begin{equation}\label{eqn9.2}
\left|\tilde{v}^{k}(\tilde{x})-\tilde{v}^{k}(x_{s})\right|<2\varepsilon\mu. 
\end{equation}
Combining \eqref{eqn2.7} and \eqref{eqn9.2}, it follows that 
\[
\widetilde{V}^{k}-\tilde{v}^{k}(\tilde{x})<5\varepsilon\mu,\ \ \text{where}\ \tilde{x}\in B(y,\delta R),
\]
which contradicts \eqref{eqn2.6}. 
This completes the proof of Theorem \ref{lem2.1}.
\end{proof}

\begin{corollary}\label{lem2.2}
Under Assumption \textsl{\ref{lem12}}, there exists a constant $\delta_{0}\in(0,1/2)$ such that for any $B(y,R)\subset\Omega_{*}$, 
\[
\operatorname{diam}_{\gamma_{*}}(\psi(B(y,\delta_{0}R)))\leq\frac{1}{2}\max\big\{\!\operatorname{diam}_{\gamma_{*}}(\psi(B(y,R))),2R\big\},  
\] 
where $\delta_{0}$ depends only on $n$, $\psi(\overline{\Omega_{*}})$, $\gamma_{*}$, $\sigma_{*}$, $\nu$ and $\operatorname{diam}\Omega_{*}$,
\end{corollary}
\begin{proof}
If $\operatorname{diam}_{\gamma_{*}}(\psi(B(y,R)))<R$, the result follows directly.
Hence, we may assume 
\[
\operatorname{diam}_{\gamma_{*}}(\psi(B(y,R)))\geq R.
\]
Let $\varepsilon_{0}$ and $\delta$ be constants given in Theorem \ref{lem2.1}, and define
\[
\mu:=\frac{1}{2}\operatorname{diam}_{\gamma_{*}}(\psi(B(y,R))).
\]
By arguments similar to those in \eqref{claim3.1}, $\psi(B(y,R))$ can be covered by $N''$ balls $\mathcal{B}_{\gamma_{*}}^{1},\dots,\mathcal{B}_{\gamma_{*}}^{N''}$ of radius $\varepsilon_{0}\mu$ with respect to the distance $\operatorname{dist}_{\gamma_{*}}(\cdot,\cdot)$, where $N''$ depends only on $n$,  $\psi(\overline{\Omega_{*}})$, $\gamma_{*}$ and $\varepsilon_{0}$.
It follows from Theorem \ref{lem2.1} that $\psi(B(y,\delta R))$ can be covered by $N''-1$ balls among $\mathcal{B}_{\gamma_{*}}^{1},\dots,\mathcal{B}_{\gamma_{*}}^{N''}$.

If $\operatorname{diam}_{\gamma_{*}}(\psi(B(y,\delta R)))\leq\mu$, we are done.
If $\operatorname{diam}_{\gamma_{*}}(\psi(B(y,\delta R)))>\mu$, then applying Theorem \ref{lem2.1} again yields that $\psi(B(y,\delta^{2}R))$ can be covered by $N''-2$ balls from $\mathcal{B}_{\gamma_{*}}^{1},\dots,\mathcal{B}_{\gamma_{*}}^{N''}$.
Since $\varepsilon_{0}\leq1/(128\sqrt{n})$, we cannot exhaust all the balls.
Therefore, we get the desired constant $\delta_{0}$ $(\geq\delta^{N''-1})$ by successive applications of Theorem \ref{lem2.1}, where $\delta_{0}$ depends only on $n$, $\psi(\overline{\Omega_{*}})$, $\gamma_{*}$, $\sigma_{*}$, $\nu$ and $\operatorname{diam}\Omega_{*}$. 
This completes the proof of Corollary \ref{lem2.2}.
\end{proof}

\begin{theorem}\label{thm1.1}
Under Assumption \textsl{\ref{lem12}}, for any $\Omega_{*}'\subset\subset\Omega_{*}$ we have H\"older seminorm estimates
\[
[\psi]_{\alpha;\Omega_{*}'}\leq Cd^{-\alpha},
\]
where $d=\operatorname{dist}(\Omega_{*}',\partial\Omega_{*})$,  $C=C\big(n,\|\psi\|_{C^{0}(\overline{\Omega_{*}})},\operatorname{diam}\Omega_{*}\big)$, and $\alpha\in(0,1)$ is a constant depending only on $n$, $\psi(\overline{\Omega_{*}})$, $\gamma_{*}$, $\sigma_{*}$, $\nu$ and $\operatorname{diam}\Omega_{*}$.
\end{theorem}
\begin{proof}
For convenience, we abbreviate $B(y,R)$ by $B_{R}$ for all $R\in(0,d]$, where $y$ is any fixed point in $\Omega_{*}'$. 
For any $R\leq d$, select $m\in\mathbb{N}$ such that 
\begin{equation}\label{eqn6.1}
\delta_{0}^{m}d<R\leq\delta_{0}^{m-1}d,
\end{equation}
where $\delta_{0}$ is given in Corollary \ref{lem2.2}. 
It follows from Corollary \ref{lem2.2} that 
\[
\operatorname{diam}_{\gamma_{*}}(\psi(B_{R}))\leq\operatorname{diam}_{\gamma_{*}}\big(\psi\big(B_{\delta_{0}^{m-1}d}\big)\big)
\leq\frac{1}{2}\max\left\{\operatorname{diam}_{\gamma_{*}}\big(\psi\big(B_{\delta_{0}^{m-2}d}\big)\big),2\delta_{0}^{m-2}d\right\}.
\]
Applying Corollary \ref{lem2.2} we get the iteration relation that for any $k\in\{1,\dots,m-2\}$, 
\begin{align*}
\max\left\{\operatorname{diam}_{\gamma_{*}}\big(\psi\big(B_{\delta_{0}^{k}d}\big)\big),2\delta_{0}^{k}d\right\}
&\leq\max\left\{\frac{1}{2}\max\left\{\operatorname{diam}_{\gamma_{*}}\big(\psi\big(B_{\delta_{0}^{k-1}d}\big)\big),2\delta_{0}^{k-1}d\right\}
,2\delta_{0}^{k}d\right\}\\
&=\frac{1}{2}\max\left\{\operatorname{diam}_{\gamma_{*}}\big(\psi\big(B_{\delta_{0}^{k-1}d}\big)\big),2\delta_{0}^{k-1}d,4\delta_{0}^{k}d\right\}\\
&=\frac{1}{2}\max\left\{\operatorname{diam}_{\gamma_{*}}\big(\psi\big(B_{\delta_{0}^{k-1}d}\big)\big),2\delta_{0}^{k-1}d\right\}
\end{align*}
by the fact $\delta_{0}<1/2$. 
Thus 
\begin{equation}\label{eqn6.2}
\begin{split}
\operatorname{diam}_{\gamma_{*}}(\psi(B_{R}))&\leq\frac{1}{2^{m-1}} \max\left\{\operatorname{diam}_{\gamma_{*}}\big(\psi\big(B_{d}\big)\big),2d\right\}\\
&=\left(\delta_{0}^{m}\right)^{\log_{\delta_{0}}\hspace{-0.8mm}\frac{1}{2}} \max\left\{2\operatorname{diam}_{\gamma_{*}}\big(\psi\big(B_{d}\big)\big),4d\right\}\\
&\leq \left(\delta_{0}^{m}\right)^{\log_{\delta_{0}}\hspace{-0.8mm}\frac{1}{2}} \left(2\big(8nM^{2}+2M^{2}\big)+4\operatorname{diam}\Omega_{*}\right). 
\end{split}
\end{equation}
Define 
\[
\alpha:=\log_{\delta_{0}}\hspace{-0.8mm}\frac{1}{2},
\]
a constant depending on $n$, $\psi(\overline{\Omega_{*}})$, $\gamma_{*}$, $\sigma_{*}$, $\nu$ and $\operatorname{diam}\Omega_{*}$. 
By virtue of \eqref{eqn6.1} and \eqref{eqn6.2}, we establish the key estimate
\begin{equation}\label{eqn2.22}
\operatorname{diam}_{\gamma_{*}}(\psi(B_{R}))\leq \left(16nM^{2}+4M^{2}+4\operatorname{diam}\Omega_{*}\right)d^{-\alpha}R^{\alpha}.
\end{equation}
Let $k\in\{1,\dots,n\}$ be the index satisfying 
\[
\mathop{\operatorname{osc}}_{B_{R}}\psi^{k}=\max_{1\leq l\leq n}\mathop{\operatorname{osc}}_{B_{R}}\psi^{l}.
\]
From \eqref{eqn2.111} and the condition $\gamma_{*}\geq4nM$, we derive  
\begin{equation}\label{eqn2.23}
\operatorname{diam}_{\gamma_{*}}(\psi(B_{R}))\geq \gamma_{*}\mathop{\operatorname{osc}}_{B_{R}}\psi^{k}-2M\sum_{l=1}^{n}\mathop{\operatorname{osc}}_{B_{R}}\psi^{l}\geq 2nM\max_{1\leq l\leq n}\mathop{\operatorname{osc}}_{B_{R}}\psi^{l}.
\end{equation}
Therefore, the combination of \eqref{eqn2.22} and \eqref{eqn2.23} yields 
\[
\mathop{\operatorname{osc}}_{B(y,R)}\psi^{l}\leq \frac{8nM^{2}+2M^{2}+2\operatorname{diam}\Omega_{*}}{nM}d^{-\alpha}R^{\alpha},\ \ \forall R\leq d,\ \forall1\leq l\leq n,\ \forall y\in\Omega_{*}'.
\]
This completes the proof of Theorem \ref{thm1.1}.
\end{proof}

\section{Quasilinear elliptic equation}

\begin{proof}[Proof of Theorem \ref{thm1.11}] 
Define
\[
v^{k,\gamma}(x):=\gamma D_{k}u(x)+|Du(x)|^{2},\ \ \forall x\in\Omega,\  \ \forall\gamma\in\mathbb{R},\ k\in\{1,\dots,n\},
\]
and for convenience, we abbreviate $v^{k,\gamma}$ by $v$.
Then it follows from Chapter 13.3 of \cite{GT} that $v$ satisfies in the generalized sense
\begin{equation}\label{eqn2.5}
D_{i}\left(a^{ij}D_{j}v\right)\geq g+D_{i}f^{i}\ \ \text{in}\ \Omega,
\end{equation}
where
\begin{align*}
&a^{ij}(x):=e^{2\chi v(x)}\hspace{0.1mm}A^{ij}(x,u(x),Du(x)),\\
&f^{i}(x):=-e^{2\chi v(x)}\hspace{0.1mm}\tilde{f}^{i}(x),\\
&g(x):=\frac{-e^{2\chi v(x)}}{\lambda(x,u(x),Du(x))}\left(\chi\sum_{i}\left|\tilde{f}^{i}(x)\right|^{2}+\sum_{i,j}\left|B^{ij}(x)\right|^{2}+\sum_{j}\left|\tilde{g}^{j}(x)\right|^{2}\right),
\end{align*}
and 
\begin{align*}
&\chi:=\sup_{x\in\Omega}\left(1+\frac{1}{\lambda^{2}(x,u(x),Du(x))}\sum_{i,j,l}\left|\left(D_{p_{l}}A^{ij}-D_{p_{j}}A^{il}\right)(x,u(x),Du(x))\right|^{2}\right),\\
&\tilde{f}^{i}(x):=\left(2D_{i}u(x)+\gamma\delta^{ik}\right)B(x,u(x),Du(x)),\\
&B^{ij}(x):=\left[\left(2D_{l}u(x)+\gamma\delta^{lk}\right)\delta_{l}A^{ij}-2\delta^{ij}B\right](x,u(x),Du(x)),\\
&\tilde{g}^{j}(x):=\delta_{i}A^{ij}(x,u(x),Du(x)),
\end{align*}
where the differential operator $\delta_{l}$ is given by 
\[
\delta_{l}h(x,z,p):=D_{x_{l}}h(x,z,p)+p_{l}D_{z}h(x,z,p).
\]
By virtue of \eqref{eqn9} and \eqref{eqn2.5} we get Assumption \ref{lem12}, which together with Theorem \ref{thm1.1} yields Theorem \ref{thm1.11}.
\end{proof}
\par
\vspace{2mm}
For boundary $C^{1,\alpha}$ estimates, using the flattening mapping, we need only consider the neighborhood of a flat boundary portion.
For any $k\in\{1,\dots,n-1\}$ and $\gamma\in\mathbb{R}$, it follows from Chapter 13.4 of \cite{GT} that 
\[
\gamma D_{k}u+\sum_{i=1}^{n-1}\left|D_{i}u\right|^{2}\ \ \text{is a subsolution of a linear equation in divergence form.}
\]
Applying the boundary weak Harnack inequality, we conclude from Section 3 that  the boundary H\"older estimate holds for $D_{i}u,\ 1\leq i\leq n-1$.
The remaining estimate of $D_{n}u$ follows from the same argument as those in Chapter 13.4 of \cite{GT}.

\section{Fully nonlinear elliptic equation}

Note that for quasilinear equations, there also exists the weak Harnack inequality, namely Theorem 9 of \cite{Trudinger1980}.
Thus replacing Assumption \ref{lem12} and \eqref{eqn2.10} with Assumption \ref{lem13} (provided below) and the weak Harnack inequality \cite{Trudinger1980} in the argument of Section 3 yields a result corresponding to Theorem \ref{thm1.1}.

\begin{assumption}\label{lem13}
Let $\psi(x)$ be an $n$-dimensional vector-valued function in $\Omega$, and $\psi^{k}(x)$ denote its $k$-th component for all $1\leq k\leq n$.
Assume that $\psi\in W^{2,n}(\Omega)$, and that for any $k\in\{1,\dots,n\}$, the functions 
\[
v_{\pm}^{k}(x):=\pm\gamma\psi^{k}(x)+|\psi(x)|^{2}\ \ \text{in}\ \overline{\Omega}
\]
are subsolutions of equations \eqref{eqn1} for some $\gamma\geq 4n\max\{\sup_{\Omega}|\psi|,1\}$.  
Furthermore, assume the coefficients of \eqref{eqn1} satisfy the following conditions in $\Omega\times\mathbb{R}\times\mathbb{R}^{n}$:  
\[
A^{ij},\ B\ \text{are measurable};\ \ \frac{\Lambda}{\lambda}\leq\sigma;\ \ B(x,z,p)\leq\lambda\left(b_{0}|p|^{2}+b|p|+g\right),
\]
where $\lambda$ and $\Lambda$ denote the minimal and maximal eigenvalues of the positive definite matrix $[A^{ij}]$, respectively; $\sigma$ and $b_{0}$ are non-negative constants independent of $k$; and $b\in L^{2n}(\Omega)$, $g\in L^{n}(\Omega)$ are non-negative functions independent of $k$.  
\end{assumption}

For fully nonlinear elliptic equations 
\begin{equation}\label{eqn2}
F\big(x,u,Du,D^{2}u\big)=0\ \ \text{in}\ \Omega,
\end{equation}
assume that there exist continuous functions $\theta_{2}$ and $\Theta_{2}$ such that  for any constant $\rho>0$,
\begin{equation}\label{eqn7}
\frac{\Lambda}{\lambda}\leq\theta_{2}(\rho)\ \ \text{and}\ \ \frac{\left|F_{x}\right|+\left|F_{z}\right|+\left|F_{p}\right|}{\lambda}(x,z,p,r)\leq\Theta_{2}(\rho)(1+|r|)
\end{equation}
in $\left\{(x,z,p,r)\in\Omega\times\mathbb{R}\times\mathbb{R}^{n}\times\mathscr{S}^{n}:\ |z|+|p|\leq\rho\right\}$, where $\lambda$ and $\Lambda$ denote the minimal and maximal eigenvalues of the positive definite matrix $[F^{ij}]$, respectively, and  $\mathscr{S}^{n}$ denotes the space of real symmetric $n\times n$ matrices.
The following theorem was established in \cite{Trudinger1983}, and below we present a new proof of this result.

\begin{theorem}\label{thm1.12}
Let $u\in W^{3,n}(\Omega)$ satisfy the equation \eqref{eqn2} with $F\in C^{1}(\Omega\times\mathbb{R}\times\mathbb{R}^{n}\times\mathscr{S}^{n})$, and let the condition \eqref{eqn7} hold. 
Then for any $\Omega'\subset\subset\Omega$ we have H\"older seminorm estimates
\[
[Du]_{\alpha;\Omega'}\leq Cd^{-\alpha},
\]
where $d=\operatorname{dist}(\Omega',\partial\Omega)$, $\alpha=\alpha\big(n,\|u\|_{C^{1}(\overline{\Omega})},\theta_{2},\Theta_{2},\operatorname{diam}\Omega\big)$,  $C=C\big(n,\|Du\|_{C^{0}(\overline{\Omega})},\operatorname{diam}\Omega\big)$.
\end{theorem}
\begin{remark}
By the Sobolev embedding theorem, $u\in W^{3,n}(\Omega)$ implies $u\in C^{1}(\overline{\Omega})$.
\end{remark}

\begin{proof}[Proof of Theorem \ref{thm1.12}]
Define
\[
v^{k,\gamma}(x):=\gamma D_{k}u(x)+|Du(x)|^{2},\ \ \forall x\in\overline{\Omega},\  \ \forall\gamma\in\mathbb{R},\ k\in\{1,\dots,n\}.
\]
By virtue of \eqref{eqn7} we get from \cite{Trudinger1983} that $v^{k,\gamma}$ satisfies 
\begin{equation}\label{eqn8}
\widetilde{F}^{ij}(x)D_{ij}v^{k,\gamma}+b_{0}\tilde{\lambda}(x)\left(\left|Dv^{k,\gamma}\right|^{2}+1\right)\geq0\ \ \text{a.e. in}\ \Omega,
\end{equation}
where $\widetilde{F}^{ij}(x):=F^{ij}(x,u(x),Du(x),D^{2}u(x))$, $\tilde{\lambda}(x)$ is the minimal eigenvalue of $[\widetilde{F}^{ij}(x)]$, and the constant $b_{0}\geq0$ depends on $\gamma$, $\|Du\|_{C^{0}(\overline{\Omega})}$, $\Theta_{2}\big(\|u\|_{C^{1}(\overline{\Omega})}\big)$.
It follows from \eqref{eqn7} and \eqref{eqn8} that Assumption \ref{lem13} holds.
Therefore, we obtain Theorem \ref{thm1.12} by applying the result corresponding to Theorem \ref{thm1.1}.
\end{proof}
\par
\vspace{2mm}
There is a geometrical proof of an interior $C^{1,\alpha}$ estimate for Lipschitz viscosity solutions given by Caffarelli and Wang \cite{CW}, in the spirit of the regularity theory for free boundaries.

\section*{Acknowledgements}

The research is supported by National Natural Science Foundation of China (Grant No. 12471202), and Shaanxi Fundamental Science Research Project for Mathematics and Physics (Grant No. 22JSZ003).

\end{document}